\documentclass[reqno,english]{amsart}

\usepackage[letterpaper,margin=1in]{geometry}
\usepackage{amsmath,amsfonts,amssymb,amsthm,mathtools}
\usepackage{algorithm}
\usepackage{algpseudocode}
\usepackage{xcolor}
\usepackage{microtype}
\usepackage[numbers]{natbib}
\usepackage{aliascnt}
\usepackage[colorlinks=true,linkcolor=blue,citecolor=blue,urlcolor=cyan]{hyperref}
\usepackage[capitalize,noabbrev]{cleveref}
\hypersetup{
    pdftitle={Sharp threshold for online balancing of i.i.d. binary vectors},
    pdfauthor={Dylan J. Altschuler and Konstantin Tikhomirov}
}

\newcommand*{\RR}{\mathbb R}
\newcommand*{\ZZ}{\mathbb Z}
\newcommand*{\PP}{\mathbb P}
\newcommand*{\EE}{\mathbb E}
\newcommand*{\calF}{\mathcal F}
\newcommand*{\calM}{\mathcal M}
\newcommand*{\ind}{\mathbf 1}
\newcommand*{\supp}{\operatorname{supp}}
\newcommand*{\sign}{\operatorname{sign}}
\newcommand*{\col}{\operatorname{col}}
\newcommand*{\Var}{\operatorname{Var}}
\newcommand*{\disc}{\operatorname{disc}}

\newtheorem{theorem}{Theorem}[section]

\newaliascnt{proposition}{theorem}
\newtheorem{proposition}[proposition]{Proposition}
\aliascntresetthe{proposition}

\newaliascnt{corollary}{theorem}

\aliascntresetthe{corollary}

\newaliascnt{lemma}{theorem}
\newtheorem{lemma}[lemma]{Lemma}
\aliascntresetthe{lemma}

\theoremstyle{definition}
\newaliascnt{definition}{theorem}
\newtheorem{definition}[definition]{Definition}
\aliascntresetthe{definition}
\newaliascnt{remark}{theorem}
\newtheorem{remark}[remark]{Remark}
\aliascntresetthe{remark}

\makeatletter
\let\c@algorithm\c@theorem

\makeatother

\title{The threshold for online balancing of i.i.d. binary vectors}

\author[D. J. Altschuler]{Dylan J. Altschuler}
\address{Department of Mathematics, The University of Texas at Austin}
\email{dylan.altschuler@austin.utexas.edu}
\author[K. Tikhomirov]{Konstantin Tikhomirov}
\address{Department of Mathematical Sciences, Carnegie Mellon University}
\email{ktikhomi@andrew.cmu.edu}

\begin{document}

\begin{abstract}
Consider the task of online vector balancing for stochastic arrivals
$X_1,\ldots,{X_T}$, where the $X_i$ are independent uniformly random
$d$--sparse binary vectors in $\{0,1\}^n$. This is a random analogue of the online Beck--Fiala problem.  We show that uniformly for
$2\le d\le n/2$ and $T = \Theta(n)$, the optimal online prefix discrepancy
$\max\limits_{t\leq T}\left\|\sum_{i=1}^t\sigma_i X_i\right\|_\infty$
is of order
\[
    \Theta\big(\max\{\sqrt d,\log\log n\}\big).
\]
The upper bound is achieved by an efficient online
algorithm. Thus, for $d\le(\log\log n)^2$, the optimal discrepancy is
$\Theta(\log\log n)$ and is independent of the sparsity up to constant
factors, whereas above this scale it is $\Theta(\sqrt d)$, matching the order
of the offline discrepancy.  This identifies the threshold at which sparsity
begins to govern the online discrepancy of the random Beck--Fiala model. \\

\textit{This work supersedes and replaces a previous article of the authors with the same lower bound, but the upper bound for $d \le (\log \log n)^2/\log\log\log n$} \cite{AT-old}.
\end{abstract}

\maketitle

\section{Introduction}
Vector balancing is a fundamental task in combinatorics with numerous applications in algorithm design and optimization, ranging from rounding integer programs to experimental design \cite{beck-fiala,experimental-design}. A vector balancing problem is the task of assigning signs $\sigma_i \in \{-1,+1\}$ to vectors $X_i \in \RR^n$ with the goal of minimizing $\|\sum X_i\sigma_i\|_\infty$. Combinatorial discrepancy, a classical and widely studied quantity, corresponds to the case that all of the vectors $X_i$ are known beforehand:
\[
    \disc(X_1,\dots,X_T) = \min_{\sigma \in \{-1,1\}^T}\Big\|\sum_{i \in [T]} X_i \sigma_i\Big\|_\infty \,.
\]
Another classical setting, of particular interest in applications, is \textit{online} discrepancy minimization \cite{spencer-online}: the $X_i$ are revealed sequentially, and $\sigma_i$ must be chosen immediately and irrevocably upon revealing $X_i$. A natural motivation for this setting is that many of the optimization problems connected to discrepancy minimization, such as bin packing and job scheduling \cite{bin}, have important online analogues. 

A rich body of works has developed estimates on the objective value for the online discrepancy minimization problem under various assumptions on the vectors $X_i$ (see, e.g., \cite{self-balancing,bansal-survey,bansal-online,optimal-online} and the references therein). Two core settings in this research direction are the so--called Koml\'os and Beck--Fiala settings, corresponding respectively to the cases that each $X_i$ is a unit vector or each $X_i$ is a $d$--sparse binary vector. The Koml\'os conjecture predicts a dimension--free discrepancy bound for unit Euclidean norm vectors, while the Beck--Fiala conjecture predicts an $O(\sqrt d)$ bound for $d$--sparse binary vectors. In contrast, the worst--case online Koml\'os discrepancy grows with the number of arrivals \cite{optimal-online}.

A line of research spanning both the online and offline settings considers random and semi--random choices of the vectors $X_i$. Random discrepancy has garnered intense recent interest for a wealth of reasons, including interest in beyond--worst--case algorithmic aspects \cite{bansal-smooth1,bansal-smooth2}, connections with statistical physics \cite{APZ}, and perturbative analogues of long--standing conjectures in discrepancy theory \cite{bansal-smooth2}. Estimating the algorithmic and existential thresholds for discrepancy of random vectors, especially with regards to distinguishing the online and offline settings, remains an active and compelling area. We refer the reader to the ICM survey of Bansal \cite{bansal-survey} for a modern, comprehensive account of all the previously mentioned topics. 

The present work investigates the online discrepancy of a sequence of uniformly random $d$--sparse binary vectors, corresponding to the {\it average-case online Beck--Fiala problem}.  Each $X_t$ is the
indicator of an independent uniformly random $d$--subset of $[n]$, or,
equivalently, the input matrix has independent uniform $d$--sparse binary
columns. For square matrices, the offline discrepancy is $\Theta(\sqrt d)$
with high probability whenever $2\leq d\leq n/2$.  The upper bound of $O(\sqrt{d})$ follows by
combining the spectral algorithm of 
\cite{potukuchi-spectral}, valid for the sparse regime of $d=o(\sqrt n)$,
with the algorithmic partial coloring result of \cite{bansal-meka}, valid for
$d=\Omega((\log\log n)^2)$. The matching offline lower bound of $\Omega(\sqrt{d})$ is 
considered in \cite{random-hypergraph-discrepancy}. 

Our main result shows that, unlike the offline regime,
in the online setting the optimal discrepancy exhibits a phase transition around $d=(\log\log n)^2$:

\begin{theorem}\label{thm:main}
There exist universal constants $c,C > 0$ such that the following holds. Let $2\le d\le n/2$, and let $A$ be an $n\times T$ random matrix with $T = n$ whose columns are
independent uniformly random $d$--sparse binary vectors.
\begin{itemize}
    \item For any fixed online algorithm, it holds with probability $1-o(1)$,
    \begin{equation}\label{eq:main-LB}
        \left\|\sum_{s\le n}\sigma_s\col_s(A)\right\|_\infty
        \ge c\big(\sqrt d+\log\log n\big).
    \end{equation}
    \item There is an efficient online algorithm for which, with probability
    $1-o(1)$,
    \begin{equation}\label{eq:main-UB}
        \max_{t\le n}\left\|\sum_{s\le t}\sigma_s\col_s(A)\right\|_\infty
        \le C\big(\sqrt d+\log\log n\big).
    \end{equation}
\end{itemize}
\end{theorem}

\begin{remark}[Extension to other time horizons]
We have presented the square case $T = n$ above for simplicity. The upper bound \eqref{eq:main-UB} immediately extends to any $T=\Theta(n)$:
for $T<n$, run the algorithm on the available prefix; for $T>n$, restart
on consecutive blocks of at most $n$ columns and add their discrepancy
bounds. The lower bound holds for any $T \ge n$, proven later in \cref{prop:lower-horizons}.
\end{remark}

Thus, when
$d\le(\log\log n)^2$, the optimal online discrepancy is
$\Theta(\log\log n)$: up to constants, it is independent of the sparsity and
hence of the Euclidean norm of each arriving vector.  When
$d\ge(\log\log n)^2$, the optimal order is instead $\Theta(\sqrt d)$, matching the known offline bounds in this regime. Note that our upper bound controls prefix discrepancy, which is stronger than controlling the terminal discrepancy. Conversely, our lower bound controls terminal discrepancy, which is stronger than lower bounding prefix discrepancy. 

Previous constructions for lower bounds on online discrepancy in the Koml\'os setting rely on block structure in addition to the usage of signed entries \cite{optimal-online}. In contrast, we consider ``mean--field'' (delocalized) matrix ensembles with binary, i.e., non--negative, entries. Our lower bound construction is conceptually different and uncovers new obstacles to achieving the offline discrepancy bound. In particular, our random construction establishes an asymptotic gap between the average online and offline Beck--Fiala settings when $d=o((\log\log n)^2)$.

Finally, we highlight a remaining open problem. While our terminal lower bound holds for every prescribed horizon $T\ge n$, the present argument does not give a matching upper bound uniformly over such horizons. 
For comparison, Bansal and Spencer \cite{bansal-spencer} establish an online
algorithm for random dense matrices of signs achieving $O(\sqrt n)$ terminal
discrepancy with high probability at any time $T := T(n)$ fixed in advance. Determining whether this is possible for the ultra-sparse case considered here remains an interesting algorithmic challenge.

\subsection{Previous work}
The study of online discrepancy was initiated by Spencer \cite{spencer-online} in 1977. We refer the reader to the ICM survey of Bansal \cite{bansal-survey} for a general account of the algorithmic developments in discrepancy theory, both offline and online.

Regarding gaps between online and offline discrepancy, a work of Kulkarni, Reis, and Rothvoss \cite{optimal-online} established the lower bound of $\Omega(\sqrt{\log T})$ for online discrepancy in the Koml\'os setting\footnote{That is, for balancing $T$ vectors of unit Euclidean norm.} (in contrast with the conjectured dimension--free discrepancy of Koml\'os matrices in the offline setting). Their lower bound for online Koml\'os is provided by constructing a sparse block--diagonal matrix populated by coupled random signs. This work also develops an algorithm---which does not run in polynomial time---that achieves discrepancy matching the lower bound, establishing the asymptotic threshold for the worst--case online Koml\'os instances. A subsequent work of Aden-Ali \cite{aden-ali-online} achieves the same bound in time linear in the input size.

Towards studying the \textit{average case} of online discrepancy minimization, Bansal et al. \cite{bansal-online} consider the general setting of i.i.d. arrivals from arbitrary distributions. Independent arrivals from dense, homogeneous distributions are quite well understood \cite{bansal-spencer,ALS2,kim-roche}. In particular, Bansal-Spencer \cite{bansal-spencer} develop the setting of uniform Rademacher vectors, corresponding to the average--case online Spencer setting, using an algorithm that has a number of similarities with ours. 

In the average case offline setting, one natural model is to take a matrix of i.i.d. random Gaussians. This model was introduced in the statistical physics literature as a ``symmetric binary perceptron'' \cite{APZ}, and its discrepancy threshold and critical window have since been determined \cite{ALS1,dja-crit,PX,ss}. At constant aspect ratio, the Gaussian model and the model with i.i.d. Rademacher entries both have discrepancy of order $\sqrt n$, as in Spencer's setting \cite{spencer1985six}.

More relevant to our sparse setting, the average case of the offline Beck--Fiala problem has been studied using sparse random binary matrices. The algorithmic results of Bansal and Meka \cite{bansal-meka} and Potukuchi \cite{potukuchi-spectral}, discussed above, establish the validity of the Beck--Fiala conjecture for this random model. Our result is in direct contrast: the Beck--Fiala conjecture fails for the {\it online} setting, even in the average case. There is also a significant body of work on the offline discrepancy of sparse binary matrices with divergingly wide aspect ratios. In suitable parameter ranges, these matrices have constant discrepancy independent of sparsity (see \cite{dja-jnw} and references therein).

A closely related two-coordinate model is the carpool or online edge orientation problem. Here each arriving edge contributes one $+1$ and one $-1$, and choosing a sign corresponds to orienting the edge. For uniformly random arrivals, Ajtai et al. \cite{ajtai-carpool} obtained a $\Theta(\log\log n)$ bound for the expected discrepancy of the greedy algorithm. There is a substantial literature on various settings of the carpool problem; for further developments, we refer to the discussions in \cite{gupta-carpool,optimal-online}.

Finally, in the offline worst--case setting, Bansal and Jiang \cite{bansal-jiang2,bansal-jiang1} established the Beck--Fiala conjecture under a sparsity lower bound: any matrix $M\in\{0,1\}^{n\times n}$ with $d$--sparse columns has discrepancy $O(\sqrt d)$ as long as $d\ge(\log n)^2$. Their upper bound is also implementable with an efficient algorithm. We subsequently proved an online $O(\sqrt d)$ bound for every fixed input sequence when $d$ is sufficiently large \cite{altschuler-tikhomirov-bf}; we use this result to cover the larger sparsities in \cref{thm:main}.

\subsection{Proof overview} 

\subsubsection{Lower bound.} We may assume $d \le (\log \log n)^2$, since we will use the known offline lower bound of $\Omega(\sqrt{d})$ outside this regime. The central object is a \emph{spread}: two disjoint large
sets of rows, henceforth called the ``spread sets'', whose running discrepancies are $\ell$ and $r$ for some $\ell\le0\le r$.  At
the start, all rows have running discrepancy of zero. Given a spread, we expose a block of
fresh columns. With high probability, many columns create ``isolated intersections'': the support of the column contains exactly one row from each spread set; moreover, these two intersecting rows do not intersect any other columns in the freshly revealed block. Each such isolated intersection forces an outward move of one of the two rows. That is, independently of the sign chosen by the online algorithm, either the row with discrepancy $r \ge 0$ moves to $r+1$, or else the row with discrepancy $\ell \le 0$ moves to $\ell-1$.

At least half of the isolated intersections lead to one of these results; say without loss of generality that at least half of the isolated intersections lead to the row with discrepancy $r$ moving to $r+1$. At the same time, because the new vectors are so sparse, many of the rows within the spread set of discrepancy $\ell$ have no support at all within any of the revealed columns and hence retain discrepancy $\ell$. In total, we find a new spread at discrepancy values $\ell$ and $r+1$, where the two new spread sets are not too small in cardinality.

The block lengths and retained spread set sizes are chosen on a doubly exponential
scale.  After $q$ stages the gap between the discrepancy of the two spread sets is at least $q$, while each of the spread sets have cardinality at least
\[
    s_q=\frac{n}{(\log n)^{4^q}}.
\]
We may take $q$ to be a constant multiple of
$\log\log n$: the construction then uses only $o(n)$ columns and retains
$n^{1-o(1)}$ rows. A key observation is that in the ultra--sparse range
$d\le(\log\log n)^2$, it holds with high probability that many rows of this final spread have no more support in any of the remaining columns, so their discrepancies are frozen
until the end of the input. This proves the $\Omega(\log\log n)$ lower bound.

\subsubsection{Upper bound} We give a unified treatment of all sparsities $2\le d\le {n^{1/5}}$. This suffices, since a recent result of the authors \cite{altschuler-tikhomirov-bf} handles the remaining regimes by giving an $O(\sqrt d)$ online prefix discrepancy upper bound for every fixed input sequence, with high probability over the algorithm's randomness.
Roughly, we assign an increasing potential
to the discrepancy of each row and choose the sign of the next column to minimize the
sum of the potentials on the arriving support. The one-row potential is capped at a specified value $Y$; once a row's potential reaches $Y$, it is referred to as ``exceptional'' and remains so thereafter. If a column contains
one of these exceptional rows, the algorithm overrides any potential-based decision and instead greedily chooses the sign which decreases
a largest exceptional discrepancy in that column.

This immediate treatment of exceptional rows is necessary. Indeed, a quirk of the ultra-sparse regime is that if we allow a significant number of rows to have discrepancy above our final target at some time $t < T$, we cannot guarantee that all of these rows will have sufficient support in the remaining columns $X_{t+1},\dots,X_T$ to correct for their excess. On the other hand, one of our core estimates is that only a vanishing fraction of rows ever become exceptional. In particular, few columns are supported on any exceptional rows, and hence the emergency greedy decisions for exceptional rows have negligible overall effect on the discrepancy of moderate rows. That is, our algorithm runs two sub-routines simultaneously---a greedy algorithm for exceptional rows and potential-based algorithm for moderate rows---whose interaction can be controlled because exceptional rows are rare. 

A detailed, extended sketch deriving the choice of potential and bounding the evolution of moderate rows is deferred to \cref{sec:potential-overview}, after the appropriate notation has been introduced. For now, let us briefly sketch how the discrepancy growth of exceptional rows is controlled. Say that a row $i$ currently has discrepancy $\Theta$, where $\Theta$ is the least possible running value of discrepancy which is considered exceptional. In order for the discrepancy of row $i$ to further increase, the next column in which row $i$ is supported must also have another row $i'$ with the same discrepancy or worse in its support. Otherwise, the discrepancy of row $i$ would have been greedily decreased. To quantify this observation, group discrepancy levels in steps of three. A row reaching $\Theta+3(r+1)$ must have made three upward moves from height at least $\Theta+3r$, each in a column shared with another row that had already attained that height. A union bound over these three columns, followed by Markov's inequality, gives a doubly exponential decay of the number of rows reaching successive levels. Specifically, our potential analysis bounds the initial exceptional fraction by $K/[d^3(\log\log n)^6]$ for an absolute constant $K$. For a sufficiently large absolute constant $D$, the recursion then bounds, with high probability, the fraction of rows ever reaching $\Theta+3r$ by
\[
    \frac{x_0^{2^r}}{Dd^3},
    \qquad x_0:=\frac{DK}{(\log\log n)^6}.
\]
For some $r\le C\log\log n$, this fraction is less than $1/n$. Thus no row reaches $\Theta+3r$, giving the required discrepancy bound since $\Theta\le C(\sqrt d+\log\log n)$. 

\begin{remark}[Comparison with previous algorithms]
    The idea of ``superimposing'' two algorithms was previously utilized by Bansal--Spencer for balancing uniform $\{-1,+1\}^n$ vectors \cite{bansal-spencer}. There, potential minimization and majority vote are alternated at each time step. Our algorithm is morally an adaptation of these ideas to the ultra--sparse regime, but the actual implementations and analysis are substantially different.
\end{remark}

\subsection{Model and notation}

Let $\mu_d$ be the uniform measure on the binary vectors in $\{0,1\}^n$ with
exactly $d$ nonzero coordinates.  We write
\[
    A\sim\calM_{n,T,d}
\]
when the columns of the $n\times T$ matrix $A$ are independent with law
$\mu_d$.  Let
\[
    R_t:=\supp(\col_t(A)).
\]
Thus $R_t$ is a uniform $d$--subset of $[n]$.  For signs
$\sigma_1,\ldots,\sigma_t$, define
\begin{equation}\label{eq:running-discrepancy}
    S_i(t):=\sum_{s\le t}\sigma_s\ind_{\{i\in R_s\}},
    \qquad S_i(0):=0.
\end{equation}

An online algorithm may use an auxiliary random seed independent of $A$, and
$\sigma_t$ must be measurable with respect to that seed and
$R_1,\ldots,R_t$.  We write $\calF_t$
for the $\sigma$--field generated by $R_1,\ldots,R_t$
and the random seed used by the algorithm.

\subsection*{Funding}
K.T. was partially supported by NSF grant DMS 2452120.

\subsection*{Comparison with superseded preprint}
The current manuscript supersedes and replaces a preprint of the authors posted in September 2025 \cite{AT-old}. The previous manuscript proved the lower bound in \cref{thm:main} (for all sparsities), as well as the sharp upper bound for the ultra-sparse regime $2 \le d \le (\log\log n)^2/\log\log\log n$. In this regime of $d$ considered in the upper bound, only exceptional rows needed to be handled. The lower bound and the treatment of exceptional rows in the upper bound retain the arguments of the previous version, with changes to the exposition. The main mathematical addition of the present version is the treatment of moderate rows in the upper bound, which enables a sharp upper bound throughout the range $2\le d\le {n^{1/5}}$. Combined with \cite{altschuler-tikhomirov-bf}, this gives the full range in \cref{thm:main}. 

\subsection*{AI acknowledgment}
The previous, superseded draft was written entirely without AI. This corresponds to the current discrepancy lower bound as well as the upper bound for the ultra-sparse regime \cite{AT-old}. Subsequently, also without AI, we extended the upper bound to the full range of sparsity considered in the current paper under the simplifying assumption that the non-zero entries of incoming vectors are independent and uniform on $\{-1,1\}$, rather than deterministically $1$. Our argument utilized a modification of the multiscale majority algorithm \cite{ALS2,kim-roche} to handle the evolution of moderate rows. 

The authors then used ChatGPT Pro 5.5 to develop and check the manuscript. During this writing phase, GPT made two substantive contributions. First, it showed us that our analysis could be readily modified to drop the random signing assumption, specifically via \cref{lem:linear}. Second, it subsequently suggested the current presentation of the potential method for controlling moderate rows. While this is equivalent to our original multiscale algorithm mathematically, the resulting exposition is significantly cleaner.

\section{Lower bound}\label{sec:lower}

This section develops the proof of the lower bound \eqref{eq:main-LB}. For $d \ge (\log\log n)^2$, the claimed discrepancy lower bound follows immediately from an existing lower bound of $\Omega(\sqrt{d})$ on the \textit{offline} discrepancy of random low-degree set systems \cite{random-hypergraph-discrepancy}, obtained via a first moment method computation. Hence, the content of this section is to prove the discrepancy lower bound of $\Omega(\log\log n)$ in the sparsity range
$2 \le d\le(\log\log n)^2$. This is the main result of the section.

As described in the overview, the idea of the proof is to use the following object as obstructions:

\begin{definition}
Let $\ell\le0\le r$ and let $k,s$ be nonnegative integers.  We say that
$(A,\sigma)$ contains an $(\ell,r)$--spread\footnote{ This definition is internal to the present article and unrelated to notions of ``spread'' in the literature of thresholds.} of size $s$ at time $k$ if there
are disjoint sets $I^{(\ell)},I^{(r)}\subset[n]$, both of size $s$, such that
\[
    S_i(k)=\ell\quad(i\in I^{(\ell)}),
    \qquad
    S_i(k)=r\quad(i\in I^{(r)}),
\]
where the {partial} sums $S_i$ are defined by \eqref{eq:running-discrepancy}.
\end{definition}

Fix $C_0:=4$.  For $q\ge0$, let
\begin{equation}\label{eq:lower-parameters}
    s_q:=\left\lceil\frac{n}{(\log n)^{C_0^q}}\right\rceil,
    \qquad
    k_q:=\sum_{u=1}^q
       \left\lceil\frac{n}{(\log n)^{e^{u-1}}}\right\rceil,
    \qquad k_0:=0.
\end{equation}

\begin{lemma}[Growth of a spread]\label{lem:spread}
Fix any online algorithm and let
\[
    0\le q\le q_*:=\left\lfloor
       \frac{\log\log n}{2\log C_0}\right\rfloor.
\]
With probability $1-O(qn^{-0.9})$, there are integers
$\ell\le0\le r$ with $r-\ell\ge q$ such that $(A,\sigma)$ contains an
$(\ell,r)$--spread of size $s_q$ at time $k_q$.
\end{lemma}

For the matrix $A$, say that a row ``occurs'' within a set of columns if it {has a nonzero entry in at least one of those columns}. 

\begin{lemma}\label{lem:block-estimate}
Assume $d\le(\log\log n)^2$. Let $m$ and $s$ satisfy $m,s \in [n^{0.99},\, 2n/\log n]$, and let $0\le t\le n-m$. Then, consider any $I,J \subset [n]$ which are disjoint
$\calF_t$--measurable random sets satisfying $|I|=|J|=s$ almost surely. Writing $A_u:=\col_u(A)$, consider the block
$$
    A_{t+1},\ldots,A_{t+m}.
$$
Conditionally on $\calF_t$, with probability $1-O(n^{-0.9})$, both of the
following hold:
\begin{enumerate}
\item at least $s/2$ rows of each of $I$ and $J$ do not occur in the
block;
\item at least
\(        \frac1{100}m\left(\frac{s}{n}\right)^2
   \)
columns contain exactly one row of $I$ and one row of $J$, and each of
these two rows occurs in no other column of the block.
\end{enumerate}
\end{lemma}
\begin{proof}
Condition on $\calF_t$. This fixes $I$ and $J$, while
$A_{t+1},\ldots,A_{t+m}$ are independent and uniformly random $d$--sparse
binary columns. All probabilities and expectations below are with respect to this conditioning.

For $i\in I\cup J$, let $X_i$ be the indicator that row $i$ does not occur in
any column of the block. Then

$$
    p:=\EE X_i=\left(1-\frac dn\right)^m=1-o(1).
$$

If $i\ne j$, independence of the columns gives
\begin{align*}
\EE(X_iX_j)
&=\left(\frac{\binom{n-2}{d}}{\binom nd}\right)^m
=\left(\frac{(n-d)(n-d-1)}{n(n-1)}\right)^m \\
&{\le} \left(\frac{n-d}{n}\right)^{2m}
=(\EE X_i)(\EE X_j).
\end{align*}
Thus the $X_i$ have pairwise nonpositive covariances. Define $Z_I:=\sum_{i\in I}X_i$. Then

$$
    \EE Z_I=sp=(1-o(1))s,
    \qquad
    \Var(Z_I)\le\sum_{i\in I}\Var(X_i)\le s.
$$
Also define $Z_J = {\sum_{j \in J}X_j}$. For all sufficiently large $n$, $\EE Z_I\ge3s/4$, and hence Chebyshev's
inequality yields

$$
    \PP\{Z_I<s/2\}
    \le \PP\{|Z_I-\EE Z_I|>s/4\}
    \le \frac{16}{s}=O(n^{-.99}).
$$
A symmetric argument gives that $\PP\{Z_J < s/2\} = O(n^{-.99})$ as well. This establishes {the lemma's first assertion}.

For $v\in{\{t+1,\ldots,t+m\}}$, let $Y_v$ indicate the event in the second
assertion, namely that exactly one row from each $I$ and $J$ are supported within $A_v$, and no other columns of the block. We compute
$$
    \EE Y_v
    =\frac{s^2\binom{n-2s}{d-2}}{\binom nd}
      \left(\frac{\binom{n-2}{d}}{\binom nd}\right)^{m-1}
    =(1+o(1))d(d-1)\left(\frac{s}{n}\right)^2.
$$

The estimate is uniform in the stated range because
$d(s+m)/n=o(1)$. Note that, in particular,

$$
    \EE\sum_{v=t+1}^{t+m}Y_v
    \ge m(s/n)^2\ge n^{0.9}
$$

for all sufficiently large $n$. Further, for distinct column indices
$v,w\in{\{t+1,\ldots,t+m\}}$, the same counting gives
\begin{align*}
\EE(Y_vY_w)
&=
\frac{s^2\binom{n-2s}{d-2}}{\binom nd}
\frac{(s-1)^2\binom{n-2s}{d-2}}{\binom nd}
\left(\frac{\binom{n-4}{d}}{\binom nd}\right)^{m-2} \\
&={(\EE Y_v)(\EE Y_w)}
\left(1+O\left(\frac1s+\frac dn+\frac{d^2m}{n^2}\right)\right).
\end{align*}
Consider
$$
    Z:=\sum_{v=t+1}^{t+m}Y_v,
    \qquad
    \mu:=\EE Z.
$$
The preceding estimates imply

$$
    \frac{\Var(Z)}{\mu^2}
    \le \frac1\mu
       +C\left(\frac1s+\frac dn+\frac{d^2m}{n^2}\right)
    =O(n^{-0.9}).
$$
Finally, for all sufficiently large $n$, note that
$    \frac1{100}m\left(\frac{s}{n}\right)^2\le\frac\mu2. 
$
Hence, by Chebyshev's inequality, 
\[
    \PP\left\{Z {<} \frac1{100}m\left(\frac{s}{n}\right)^2\right\} \le \PP\{Z {<} \mu/2\} = O(n^{-0.9}).
\]
This completes the second assertion and hence the lemma.
\end{proof}

\begin{proof}[Proof of \cref{lem:spread}]
We begin by noting that uniformly for $1\le q\le q_*$,
\begin{equation}\label{eq:spread-params} 
    n^{0.99}\le s_{q-1}\le \frac{2n}{\log n},
    \qquad
    n^{0.99}\le k_q-k_{q-1}\le \frac{2n}{\log n}.
\end{equation}

Next, for each integer $0\le u\le q_*$, let $\mathcal E_u$ be the event that there
are integers $\ell\le0\le r$ with $r-\ell\ge u$ such that $(A,\sigma)$
contains an $(\ell,r)$--spread of size $s_u$ at time $k_u$. We will prove by
induction that, for a sufficiently large absolute constant $C$,
\[
    \PP(\mathcal E_u^c)\le Cu n^{-0.9}.
\]
The base case is trivial:
at time zero all rows have discrepancy zero. Since $2s_0\le n$ for all
sufficiently large $n$, the event $\mathcal E_0$ holds surely. So, fix $1\le q\le q_*$ and assume by way of induction that
\[
    \PP(\mathcal E_{q-1}^c)\le C(q-1)n^{-0.9}.
\]
Condition on $\calF_{k_{q-1}}$ and suppose
that $\mathcal E_{q-1}$ holds. We may choose 
integers $\ell\le0\le r$ with $r - \ell \ge q-1$  and disjoint sets
$I,J\subset[n]$ witnessing an $(\ell,r)$--spread of size
$s_{q-1}$ at time $k_{q-1}$. (As there are potentially multiple valid such choices of $\ell,r,I,J$,  we choose according to some arbitrary deterministic rule, fixed ahead of time). These choices are
$\calF_{k_{q-1}}$--measurable. By \eqref{eq:spread-params}, we may apply  \cref{lem:block-estimate} with
\[
    t=k_{q-1},\qquad
    m=k_q-k_{q-1},
\]
to the block $A_{k_{q-1}+1},\ldots,A_{k_q}$. Let $\mathcal J$ be the set
of columns from the second conclusion of that lemma. In each column
$v\in\mathcal J$, let $i_v^{(\ell)} \in I$ and $i_v^{(r)} \in J$ be its selected rows.
Since these rows occur only in column $v$ during the block, either
\[
    S_{i_v^{(r)}}(k_q)=r+1
    \quad\text{or}\quad
    S_{i_v^{(\ell)}}(k_q)=\ell-1,
\]
according to the sign chosen in that column. Partition the columns accordingly:
\[
    \mathcal J_+
    :=\{v\in\mathcal J:S_{i_v^{(r)}}(k_q)=r+1\},
    \qquad
    \mathcal J_-:=\mathcal J\setminus\mathcal J_+.
\]
By the pigeonhole principle, at least one of $\mathcal J_+$ and $\mathcal J_-$ has size at least
$|\mathcal J|/2$. If $|\mathcal J_+|\ge|\mathcal J|/2$, then the distinct
rows $i_v^{(r)}$, $v\in\mathcal J_+$, all have discrepancy $r+1$ at time
$k_q$, while at least $s_{q-1}/2$ rows of $I$ have discrepancy
$\ell$ at time $k_q$. Thus $(A,\sigma)$ contains an
$(\ell,r+1)$--spread of size at least
\[
    \min\left\{\frac{s_{q-1}}2,\frac{|\mathcal J|}2\right\}.
\]
If $|\mathcal J_-|\ge|\mathcal J|/2$, the same argument gives an
$(\ell-1,r)$--spread of at least this size. In either case, the difference
between the two discrepancy values has increased by one and is therefore
at least $q$. Moreover,
\[
    \frac{|\mathcal J|}{2}
    \ge c\frac{n}
        {(\log n)^{2C_0^{q-1}+e^{q-1}}}
    \ge s_q,
\]
where the final inequality holds for all sufficiently large $n$ because
\[
    2C_0^{q-1}+e^{q-1}\le C_0^q-1.
\]
The bound $s_{q-1}/2\ge s_q$ is immediate for all sufficiently large $n$.
Consequently, whenever the conclusions of \cref{lem:block-estimate} hold,
the event $\mathcal E_q$ occurs. Thus, on $\mathcal E_{q-1}$,
\[
    \PP\big(\mathcal E_q^c\mid\calF_{k_{q-1}}\big)
    \le Cn^{-0.9}.
\]
Recalling the inductive hypothesis completes the induction:
\[
    \PP(\mathcal E_q^c)
    \le \PP(\mathcal E_{q-1}^c)+Cn^{-0.9}
    \le Cq n^{-0.9}.
\]

\end{proof}

We are ready to deduce \eqref{eq:main-LB}; the proof is recorded as the following proposition. Note that our lower bound is ``robust'' in the sense that we find many rows with large discrepancy, rather than just a single row. 

\begin{proposition}[Sparse online lower bound]\label{prop:sparse-lower}
Let $2\le d\le(\log\log n)^2$ and $A\sim\calM_{n,n,d}$.  For every online
algorithm, with probability
$1-O(n^{-0.8})$ there are more than $n^{1/2}$ rows $i$ such that
\[
    |S_i(n)|\ge c_0\log\log n,
    \qquad c_0:=\frac1{8\log4}.
\]
\end{proposition}

\begin{proof}
Apply \cref{lem:spread} with $q=q_*$. Its failure probability is at most
\[
    O(q_*n^{-0.9})=O(n^{-0.8}).
\]
On the complementary event $\mathcal E_{q_*}$, choose, according to a fixed
deterministic rule, integers $\ell\le0\le r$ and disjoint sets
$I,J$ witnessing the resulting spread. Since
$r-\ell\ge q_*$,
\[
    \max\{|\ell|,|r|\}\ge \frac{q_*}{2}
    \ge c_0\log\log n.
\]
So either the rows of $I$ or of $J$ have
discrepancy at time $k_{q_*}$ of at least $c_0\log\log n$; say it is $I$ without loss of generality. Then
\[
    |S_i(k_{q_*})|\ge c_0\log\log n
    \qquad\text{for every }i\in I.
\]
Since $C_0^{q_*}\le\sqrt{\log n}$,
\[
    |I|\ge \frac{n}{(\log n)^{\sqrt{\log n}}}=n^{1-o(1)}.
\]

Also $k_{q_*}=o(n)$. Let $U$ be the number of rows of $I$ which do not occur
in any of the columns after time $k_{q_*}$, and write
\[
    \mu:=\EE\big[U\mid\calF_{k_{q_*}}\big].
\]
Almost surely on $\mathcal E_{q_*}$, for all sufficiently large $n$
\[
    \mu
    =|I|\left(1-\frac dn\right)^{n-k_{q_*}}
    \ge |I|\left(1-\frac dn\right)^n
    \ge |I|e^{-2d}
    =n^{1-o(1)}
    \ge n^{0.9}.
\]

By the same calculation as in the proof of
\cref{lem:block-estimate}, conditionally on $\calF_{k_{q_*}}$ the
corresponding indicators (that each row $i\in I$ is not supported in the remaining columns of the matrix after time $k_{q_*}$) have nonpositive covariance. Hence, almost surely
on $\mathcal E_{q_*}$,
\[
    \Var\big(U\mid\calF_{k_{q_*}}\big)\le\mu.
\]
Since $n^{1/2}\le\mu/2$ for all sufficiently large $n$, Chebyshev's
inequality gives that almost surely on $\mathcal E_{q_*}$,
\[
\begin{aligned}
    \ind_{\mathcal E_{q_*}}
    \PP\big\{U\le n^{1/2}\mid\mathcal F_{k_{q_*}}\big\}
    &\le
    \ind_{\mathcal E_{q_*}}
    \PP\big\{|U-\mu|\ge\mu/2\mid\mathcal F_{k_{q_*}}\big\} \\
    &\le
    \frac{4\ind_{\mathcal E_{q_*}}
    \Var(U\mid\mathcal F_{k_{q_*}})}{\mu^2} \\
    &\le \frac{4\ind_{\mathcal E_{q_*}}}{\mu}
    \le 4n^{-0.9}\ind_{\mathcal E_{q_*}}.
\end{aligned}
\]
Taking expectations and using that
$\mathcal E_{q_*}\in\mathcal F_{k_{q_*}}$, we obtain
\[
    \PP\big(\mathcal E_{q_*}\cap\{U\le n^{1/2}\}\big)
    \le 4n^{-0.9}.
\]
On $\mathcal E_{q_*}\cap\{U>n^{1/2}\}$, more than $n^{1/2}$ rows of $I$
do not occur in any remaining column.  Their discrepancies therefore do not
change after time $k_{q_*}$, and each has terminal absolute discrepancy at
least $c_0\log\log n$.  Consequently, the desired event has probability at
least
\[
    1-\PP(\mathcal E_{q_*}^{\,c})
      -\PP\big(\mathcal E_{q_*}\cap\{U\le n^{1/2}\}\big)
    =1-O(n^{-0.8}).
\]

\end{proof}

We now extend our lower bound to hold for arbitrarily long time horizons, $T \ge n$. 
\begin{proposition}[Extension of lower bound]
\label{prop:lower-horizons}
There is an absolute constant $c>0$ such that the following holds.  Let
$2\le d\le n/2$, let $T\ge n$, and let $A\sim\calM_{n,T,d}$.  Then every
online algorithm satisfies, with probability $1-o(1)$,
\[
    \left\|\sum_{s=1}^T\sigma_s\col_s(A)\right\|_\infty
    \ge c\big(\sqrt d+\log\log n\big).
\]
In particular, \eqref{eq:main-LB} remains valid, with the summation over $[T]$ instead of $[n]$.
\end{proposition}

\begin{proof}
Let $t_0:=T-n$ and write
\[
    S(t):=(S_1(t),\ldots,S_n(t)).
\]
First suppose that $d\le(\log\log n)^2$, and condition on $\calF_{t_0}$.  The final
$n$ columns are independent of $\calF_{t_0}$, and the restriction of the
algorithm to these columns is an online algorithm whose auxiliary data
include the fixed (i.e., $\mathcal{F}_{t_0}$ measurable) preceding history.  Thus, applying
\cref{prop:sparse-lower} to the final block, it holds with conditional probability
$1-O(n^{-0.8})$ that the set
\[
    L:=\left\{i\in[n]:
       |S_i(T)-S_i(t_0)|\ge c_0\log\log n\right\}
\]
has size greater than $n^{1/2}$. Next, consider the set of rows that already incur large discrepancy at $t_0$:
\[
    L_0:=\left\{i\in[n]:
       |S_i(t_0)|\ge\frac{c_0}{2}\log\log n\right\}.
\]
If $|L_0|\le n^{1/2}$ and $|L|>n^{1/2}$, then some $i\in L\setminus L_0$,
and hence
\[
    |S_i(T)|
    \ge |S_i(T)-S_i(t_0)|-|S_i(t_0)|
    \ge\frac{c_0}{2}\log\log n.
\]

Suppose instead that $|L_0|>n^{1/2}$.  For $i\in L_0$, let $X_i$ indicate
that row $i$ does not occur in any of the final $n$ columns, and put
$U:=\sum_{i\in L_0}X_i$.  Conditional on $\calF_{t_0}$,
\[
    \mu:=\EE[U\mid\calF_{t_0}]
    =|L_0|\left(1-\frac dn\right)^n
    \ge n^{1/2}e^{-2d}
    =n^{1/2-o(1)}.
\]
Computing identically to the proof of \cref{lem:block-estimate}, the $X_i$ have conditionally
nonpositive pairwise covariances, so $\Var(U\mid\calF_{t_0})\le\mu$. Chebyshev's inequality then implies that almost surely on $\{|L_0|>n^{1/2}\}$,
\[
{\color{black}\PP\{U=0\mid\calF_{t_0}\}\le\frac1\mu=o(1).}
\]
Thus, with high probability,
some row $i$ of $L_0$ is untouched in the final block and satisfies
\[
    |S_i(T)|=|S_i(t_0)|
    \ge\frac{c_0}{2}\log\log n.
\]
Hence the terminal discrepancy is lower bounded by $\frac{c_0}{2}\log\log n$ with high probability whenever $d\le(\log\log n)^2$.

It remains to consider
$d>(\log\log n)^2$. The offline lower bound \cite{random-hypergraph-discrepancy}, valid for constant aspect ratios and $d \to \infty$,
gives for $A \sim \mathcal{M}_{n,2n,d}$ the existence of an absolute constant $c_1>0$ with
\[
    \varepsilon_n:=
    \PP\left(\disc(A)<c_1\sqrt d\right)=o(1).
\]

We first deduce from this a lower bound in the
presence of an arbitrary fixed initial discrepancy.  Let $B_1,B_2$ be
independent matrices with law $\calM_{n,n,d}$, and define
\[
    \varepsilon_n:=
    \PP\left\{\disc([B_1\;B_2])
        <c_{1}\sqrt d\right\}=o(1).
\]
For a fixed vector $x\in\RR^n$ and $\ell \in {\{1,2\}}$, let
\[
    \mathcal{E}_{x,\ell} := \left\{
       \min_{\tau\in\{-1,1\}^n}
       \|x+B_\ell\tau\|_\infty
       <\frac{c_{1}}{3}\sqrt d
    \right\}, \qquad p_x:=
    \PP(\mathcal{E}_{x,\ell}).
\]
Here, $p_x$ does not depend on $\ell$ as $B_1,B_2$ have the same distribution. Note that if $\mathcal{E}_{x,1}$ and $\mathcal{E}_{x,2}$ both hold, certified by some signings $\tau_1$ and $\tau_2$, respectively, then
the concatenation $[\tau_1,-\tau_2]\in\{-1,1\}^{2n}$ is also a signing of the columns of
$[B_1\;B_2]$ certifying:
\[
\begin{aligned}
    \disc([B_1,B_2]) \le \|B_1\tau_1-B_2\tau_2\|_\infty
    &=\|(x+B_1\tau_1)-(x+B_2\tau_2)\|_\infty < \frac{2c_1}{3} \sqrt{d}.
\end{aligned}
\]
Hence, by independence of $\mathcal{E}_{x,1}$ and $\mathcal{E}_{x,2}$, we obtain $p_x^2\le\varepsilon_n$.
The right-hand side does not depend on $x$, so
\begin{equation}\label{eq:shifted-offline-lower}
    \sup_{x\in\RR^n}p_x
    \le\sqrt{\varepsilon_n}=o(1).
\end{equation}

We now apply \eqref{eq:shifted-offline-lower} to deduce the discrepancy lower bound. Let $A \sim \mathcal{M}_{n,T,d}$ for $T \ge n$, and denote the final block of $n$ columns by $B$, and denote the signs selected by the online algorithm for $B$ by 
\[
    \rho:=(\sigma_{t_0+1},\ldots,\sigma_T).
\]
Then, 
\[
\begin{aligned}
    \left\{\|S(T)\|_\infty
       <\frac{c_1}{3}\sqrt d\right\}
    &=
    \left\{\|S(t_0)+B\rho\|_\infty
       <\frac{c_1}{3}\sqrt d\right\} \subseteq
    \left\{
       \min_{\tau\in\{-1,1\}^n}
       \|S(t_0)+B\tau\|_\infty
       <\frac{{c_1}}{3}\sqrt d
    \right\}.
\end{aligned}
\]
Hence from \eqref{eq:shifted-offline-lower}, almost surely,
\[
\begin{aligned}
    \PP\left\{
       \|S(T)\|_\infty<\frac{c_{1}}{3}\sqrt d
       \,\middle|\,\calF_{t_0}\right\}
    &\le p_{S(t_0)}
     \le\sqrt{\varepsilon_n}
     =o(1).
\end{aligned}
\]
{Taking expectations shows that, with probability $1-o(1)$,}
\begin{equation}\label{eq:long-horizon-sqrtd}
    \|S(T)\|_\infty
    \ge\frac{{c_1}}{3}\sqrt d{\color{black}.}
\end{equation}
Hence, for any fixed online algorithm, it holds with high probability whenever $d>(\log\log n)^2$ that ${\|S(T)\|_\infty} = \Omega(\sqrt{d})$. This completes the proposition.
\end{proof}

\section{Upper bound}\label{sec:potential-algorithm}

This section is dedicated to proving the following result.
\begin{proposition}[Random sparse upper bound]\label{prop:sparse-upper}
There is an absolute constant $C>0$ 
with the following property.
Let $2\le d\le n^{1/5}$, and let $A\sim\calM_{n,n,d}$. Then there is an online algorithm for selecting signs $\sigma_1,\dots,\sigma_n$
satisfying
\[
    \PP\left\{
    \max_{t\le n}\left\|\sum_{s\le t}\sigma_s\col_s(A)\right\|_\infty
    >C(\sqrt d+\log\log n)\right\}
    =O\big((\log\log n)^{-6}\big).
\]
\end{proposition}
Here, the choice $d\le n^{1/5}$ is for notational convenience; {the same argument gives the upper bound with probability $1-o(1)$ whenever $d^4(\log\log n)^6=o(n)$}. Our worst-case results in \cite{altschuler-tikhomirov-bf} cover the regime $d \ge \log(n)^{2+o(1)}$ anyway. 

\subsection{The algorithm}

Below, we assume that $a$ is
a sufficiently large absolute constant to be chosen later.
Let
\begin{equation}\label{eq:Y-def}
    Y:=d^3(\log\log n)^6,
\end{equation}
which will be used as the threshold in potential function for separating ``moderate'' and ``exceptional'' rows. Towards constructing the potential function, define
\begin{equation}\label{eq:potential-parameters}
    \rho:=a(\sqrt d+2),\qquad
    b:=\lfloor\rho-a\rfloor,\qquad
    \lambda:=2\log\frac{\rho^2-b^2}{\rho^2-(b+1)^2}.
\end{equation}
For $k\in\ZZ$, define the one-row potential by
\begin{equation}\label{eq:F-def}
F(k):=
\begin{cases}
\displaystyle
\min\!\left\{
    Y,\,
    \left(\frac{\rho^2}{\rho^2-k^2}\right)^2
\right\},
    & |k|\le b,\\[1.1em]
\displaystyle
\min\!\left\{
    Y,\,
    \left(\frac{\rho^2}{\rho^2-b^2}\right)^2
    e^{\lambda(|k|-b)}
\right\},
    & |k|>b.
\end{cases}
\end{equation}
The definition of $\lambda$ is exactly so that the two formulas match at $|k|=b+1$. As expanded upon below, both parts of the piece-wise formula for $F$ are familiar: the inverse-square potential is inspired by \cite{bansal-spencer}, while the exponential potential is {\color{black}comparable to a rescaling} of the hyperbolic-cosine potential introduced in \cite{spencer-online}.

Recalling that $Y$ was the threshold for exceptional rows in terms of potential, we also can define a corresponding equivalent threshold in terms of discrepancy:
\begin{equation}\label{eq:Theta-def}
    \Theta:=\min\{k\ge0:F(k)=Y\}.
\end{equation}

\begin{algorithm}[H]
\caption{Online balancing}\label{alg:main}
\begin{algorithmic}[1]
\State Initialize $S_i(0)=0$ for every $i\in[n]$, and set $E_0=\emptyset$.
\For{$t=1,\ldots,n$}
    \State Reveal $R_t=\supp(\col_t(A))$.
    \If{$R_t\cap E_{t-1}\ne\emptyset$}
        \State Choose
        $i_*\in\arg\max_{i\in R_t\cap E_{t-1}}|S_i(t-1)|$,
        breaking ties deterministically.
        \State Set $\sigma_t=-\sign(S_{i_*}(t-1))$, where $\sign(0):=1$.
    \Else
        \State Choose $\sigma_t$, breaking ties deterministically, according to
        \Statex
        \begin{equation}\label{eq:algorithm-potential-choice}
        \sigma_t\in\arg\min_{\sigma\in\{-1,+1\}}
            \sum_{i\in R_t}F(S_i(t-1)+\sigma).
        \end{equation}
    \EndIf
    \State Update
    $S_i(t):=S_i(t-1)+\sigma_t\ind_{\{i\in R_t\}}$
    for every $i\in[n]$.
    \State Set $E_t:=\{i:\max_{s\le t}|S_i(s)|\ge\Theta\}$.
\EndFor
\end{algorithmic}
\end{algorithm}
The sets $E_t$ will be referred to as {\it exceptional} rows.
Note that once reaching $\Theta$, a row
remains exceptional even if its discrepancy later decreases.
If a column contains an exceptional row, the
algorithm chooses the sign which decreases the absolute discrepancy
of a largest such row, unless that discrepancy is zero.
On the other hand, a column containing no exceptional row is signed to minimize the new sum of its row potentials.  

\subsection{Overview of the potential}\label{sec:potential-overview}
We now heuristically derive and analyze our choice of potential function, closely following an analogous computation of Bansal and Spencer \cite[Section~2.1]{bansal-spencer}. For an even potential $F$ and a sign $\sigma\in\{-1,+1\}$,
\begin{equation}\label{eq:guiding-decomposition}
 F(s+\sigma)-F(s)
 =\frac{F(s+1)+F(s-1)-2F(s)}2
 +\sigma\frac{F(s+1)-F(s-1)}2.
\end{equation}
These are the discrete quadratic and linear terms. For an incoming
column support $R$ at time $t$, let
\begin{align*}
 Q_R&:=\sum_{i\in R}\big(F(S_i(t-1)+1)+F(S_i(t-1)-1)-2F(S_i(t-1))\big),\\
 L_R&:=\sum_{i\in R}\big(F(S_i(t-1)+1)-F(S_i(t-1)-1)\big).
\end{align*}
If $R$ contains no exceptional row, the algorithm makes the change
$(Q_R-|L_R|)/2$. Thus the available decrease $|L_R|$ must absorb the
positive part of the ``curvature cost'' $Q_R$.

We first consider rows with absolute discrepancy $s \le b$, so that the inverse-square potential is at play. The rows with small discrepancy, say $s \le \rho/2$, {\color{black}each} contribute at most $C/(a^2d)$ to $Q_R$; and hence their total contribution to {\color{black}$Q_R$} is bounded above by $C/a^2$. Next, say that there are $M$ rows with absolute discrepancy $s \in [\rho/2, b]$. {\color{black}Denote their expected contribution to $Q_R$ and the expected absolute value of their contribution to $L_R$ by $\Delta Q_R$ and $\Delta L_R$, respectively.} Write $\delta:=\rho-s$, $u:=F(s)$, and let $m:=dM/n$
be their expected number in a fresh support.
The first and second differences have sizes $u/\delta$ and $u/\delta^2$.

For $m\ge1$, we claim
\[
 {\color{black}\Delta L_R \gtrsim\frac{u}{\delta}\sqrt m},
 \qquad
 \Delta Q_R\ \sim\frac{u}{\delta^2}m.
\]
Proving {\color{black}this lower bound}, done in \cref{lem:linear}, is one of the major technical challenges; {\color{black}it does} \textit{not} follow from a CLT argument\footnote{It would, however, if each non-zero entry of the matrix had an iid random sign, as in Bansal-Spencer \cite{bansal-spencer}. See AI disclosure.}, as $(\mathrm{sign}(S_i({\color{black}t-1})))_{i \in [n]}$ are not independent. Instead, we use symmetrization and convexity {\color{black}to obtain a bound valid for arbitrary signs of the running discrepancies}.  

If the total potential is at most $Kn$, then $Mu\le Kn$, giving
\begin{equation}\label{eq:guiding-difference-scales}
 \frac{\Delta Q_R}{\Delta L_R}
 {\color{black}\lesssim}\frac{\sqrt m}{\delta}
 \le\sqrt{\frac{Kd}{u\delta^2}}.
\end{equation}
The inverse-square formula was chosen because
\begin{equation}\label{eq:guiding-choice}
 u\delta^2=\frac{\rho^4}{(\rho+s)^2}
 \ge\frac{\rho^2}{4}\ge\frac{a^2d}{4}.
\end{equation}
The ratio in \eqref{eq:guiding-difference-scales} is therefore at most order $\sqrt K/a$. 

Finally, we consider the case $m\le1$. The probability of sampling exactly one of these rows
is of order $m$, giving $\Delta L_R \sim mu/\delta$.
Consequently,
\[
    \frac{\Delta Q_R}{\Delta L_R}
 \sim\frac{mu/\delta^2}{mu/\delta}=\frac1\delta \le \frac{1}{a}.
\]
The sampling estimate in \cref{lem:linear} is used in making all of the above analysis of the inverse-quadratic potential rigorous.

In the previous estimate, if $\delta$ approaches zero, $\Delta Q_R/\Delta L_R$ diverges, which would cause us to lose control of the change in potential. This necessitates the switch in row-potential at $b$, where
$\rho-b\in[a,a+1)$ and $cd\le F(b)\le Cd$. Up to constant factors and
rescaling, the chosen continuation is the familiar hyperbolic-cosine potential
used in online discrepancy
\cite{spencer-online,bansal-online,bansal-spencer,matrixspencer-cosine}. Indeed, for $k \in (b,\Theta-1)$, 
\[
 \frac{F(k+1)-F(k-1)}2=F(k)\sinh\lambda,
 \qquad
 \frac{F(k+1)+F(k-1)-2F(k)}2=F(k)(\cosh\lambda-1).
\]
The key point is that their ratio stays small: 
\[
 \frac{\cosh\lambda-1}{\sinh\lambda}
 =\tanh(\lambda/2)\le\frac Ca.
\]
The same sampling calculation gives that {\color{black}the ratio of expected curvature cost to expected available decrease} is at most
$C/a$ for $m\le1$, and $C\sqrt K/a$ for $1\le m\le CK$. Indeed, if the total potential is bounded by $Kn$, there
are at most $CKn/d$ rows beyond the elbow at $b$, so a fresh column
samples at most $CK$ of them on average. 
Finally, note as a sanity check that our definition of moderate row does not have excessive discrepancy: for $\lambda$ is between $c/a$ and $C/a$,
\[
 \Theta=b+\left\lceil\lambda^{-1}\log\frac{Y}{F(b)}\right\rceil
 \le Ca(\sqrt d+\log Y).
\]

It remains to bound the number of rows that ever become exceptional. We cannot simply track the potential due to the subtlety that exceptional rows stay exceptional for all time by definition, whereas their running discrepancy (and contribution to the potential) can decrease to moderate levels. Instead, we track
\[
 \Phi_t:=Y|E_t|+\sum_{i\notin E_t}F(S_i(t)),
\]
so that exceptional rows always contribute $Y$ even if their discrepancy falls,
yielding $|E_t|\le\Phi_t/Y$.

As long as $\Phi_{t-1}\le Kn$ and $a$ is sufficiently large compared with
$\sqrt K$, the preceding estimates show that the main increase in potential comes from rows with ${\color{black}|k|} \le \rho/2$, 
\[
 \EE[\Phi_t-\Phi_{t-1}\mid\calF_{t-1},R_t\cap E_{t-1}=\emptyset]
 \le \frac{C}{a^2}.
\]
A column meets $E_{t-1}$ with probability at most $Kd/Y$; conditional
on this event, its expected increase in potential ({\color{black}by the bounds on both first and second differences}) is at most
$CKd/a$. Hence
\[
 \EE[\Phi_t-\Phi_{t-1}\mid\calF_{t-1}]
 \le\frac{C}{a^2}+\frac{CK^2d^2}{aY}
 =\frac{C}{a^2}+\frac{CK^2}{ad(\log\log n)^6}.
\]
For fixed $a,K$, the last term tends to zero.
Starting from $\Phi_0=n$, we subtract this drift bound per step and stop
when $\Phi_t>Kn$. Freedman's inequality then gives, with high probability,
\[
 \max_{t\le n}\Phi_t\le Kn,
 \qquad
 |E_n|\le\frac{Kn}{Y}=\frac{Kn}{d^3(\log\log n)^6}.
\]
The factor $d^{-3}$ supplies the initial bound needed for the union
bound over triples of columns in \cref{sec:above-threshold}.

\subsection{Finite-difference estimates}
Write the centered first and second finite differences as
\begin{equation}\label{eq:finite-differences}
 \nabla F(k):=\frac{F(k+1)-F(k-1)}2,
 \qquad
 \Delta F(k):=\frac{F(k+1)+F(k-1)-2F(k)}2.
\end{equation}
Thus \eqref{eq:guiding-decomposition} becomes the exact identity
\begin{equation}\label{eq:finite-difference-identity}
    F(k+\sigma)-F(k)=\Delta F(k)+\sigma\nabla F(k).
\end{equation}

We begin with a key technical estimate giving bounds on the increase of the row potential in various regimes, making rigorous the previous heuristic estimates.

\begin{lemma}[Finite-difference estimates]\label{lem:potential}
Assuming $a$ is sufficiently large, the following estimates hold with
absolute constants $c,C>0$:
\begin{equation}\label{eq:lambda-theta-bounds}
    \frac{c}{a}\le\lambda\le\frac{C}{a},
    \qquad
    \Theta\le Ca(\sqrt d+\log Y).
\end{equation}
For every $k\in\ZZ$,
\begin{equation}\label{eq:first-difference-bound}
    |\nabla F(k)|\le\frac{C}{a}F(k).
\end{equation}
Moreover,
\begin{align}
 (\Delta F(k))_+&\le\frac{C}{a^2d},
       && |k|\le\rho/2,\label{eq:second-difference-small}\\
 (\Delta F(k))_+&\le\frac{C}{a}|\nabla F(k)|,
       && |k|>\rho/2,\label{eq:second-difference-large-1}\\
 (\Delta F(k))_+&\le
       C|\nabla F(k)|\sqrt{\frac{F(k)}{a^2d}},
       && |k|>\rho/2.\label{eq:second-difference-large-2}
\end{align}
Finally,
\begin{equation}\label{eq:second-by-first}
 |\Delta F(k)|\le|\nabla F(k)|\quad(k\ne0),
 \qquad
 |\Delta F(0)|\le\frac{C}{a^2d}.
\end{equation}
\end{lemma}

\begin{proof}
By evenness, it suffices to consider nonnegative integers $k$.
Throughout the proof, $c,C>0$ are absolute constants, and we take $a\ge4$.
Let
\[
    g(x):=\left(\frac{\rho^2}{\rho^2-x^2}\right)^2,
    \qquad |x|<\rho.
\]
Since $a\le\rho-b<a+1$, we have
\[
 \frac{1}{2a}
 \le\frac{2b+1}{(\rho-b-1)(\rho+b+1)}
 \le\frac{3}{a},
 \qquad
 \frac{d}{16}\le g(b)\le g(b+1)\le16d.
\]
The identity
\[
 \lambda=2\log\left(1+
   \frac{2b+1}{(\rho-b-1)(\rho+b+1)}\right)
\]
therefore gives $1/(2a)\le\lambda\le6/a$.
In the range $2\le d\le {n^{1/5}}$, we have $Y>16d$, so
$\Theta\ge b+2$ and
\[
 \Theta=b+\left\lceil\lambda^{-1}\log\frac{Y}{g(b)}\right\rceil
 \le Ca(\sqrt d+\log Y).
\]
This proves \eqref{eq:lambda-theta-bounds}. We now treat the inverse-square,
{\color{black}hyperbolic-cosine}, and truncated parts of the potential in turn.

For $0\le k\le b$, the three values $F(k-1),F(k),F(k+1)$ agree with
$g(k-1),g(k),g(k+1)$, including at $k=b$ because
$g(b)e^\lambda=g(b+1)$. Differentiating gives
\begin{equation}\label{eq:g-derivatives}
 \frac{g'(x)}{g(x)}=\frac{4x}{\rho^2-x^2},
 \qquad
 \frac{g''(x)}{g(x)}
   =\frac{4\rho^2+20x^2}{(\rho^2-x^2)^2}.
\end{equation}
Since $\rho-k\ge a\ge4$, these formulas imply, for $|x-k|\le1$,
\begin{equation}\label{eq:g-local-bounds}
 \frac{g(k)}{C}\le g(x)\le Cg(k),\qquad
 |g'(x)|\le\frac{Cg(k)}{\rho-k},\qquad
 0\le g''(x)\le\frac{Cg(k)}{(\rho-k)^2}.
\end{equation}
Using the integral identities
\[
 \nabla F(k)=\frac12\int_{-1}^{1}g'(k+t)\,dt,
 \qquad
 \Delta F(k)=\frac12\int_{-1}^{1}(1-|t|)g''(k+t)\,dt,
\]
we obtain
\[
 |\nabla F(k)|\le\frac{C}{a}F(k),
 \qquad
 0\le\Delta F(k)\le\frac{CF(k)}{(\rho-k)^2}.
\]
If $k\le\rho/2$, then $F(k)\le16/9$ and $\rho-k\ge\rho/2$, proving
\eqref{eq:second-difference-small}. If $\rho/2<k\le b$, then
\eqref{eq:g-derivatives} also gives
$g'(x)\ge cg(k)/(\rho-k)$ for $|x-k|\le1$. Consequently,
\[
 \Delta F(k)\le\frac{C}{\rho-k}|\nabla F(k)|.
\]
Both remaining second-difference bounds follow from
\[
 \frac{1}{\rho-k}\le\frac1a,
 \qquad
 \sqrt{F(k)}=\frac{\rho^2}{(\rho-k)(\rho+k)}
 \ge\frac{a\sqrt d}{2(\rho-k)}.
\]

For $b<k\le\Theta-2$, {\color{black}the three potential values $F(k-1),F(k),F(k+1)$ satisfy $F(k\pm1)=e^{\pm\lambda}F(k)$.} Hence
\[
 \nabla F(k)=F(k)\sinh\lambda,
 \qquad
 \Delta F(k)=F(k)(\cosh\lambda-1).
\]
Since $\lambda\le6/a$, it follows that
\[
 |\nabla F(k)|\le\frac{C}{a}F(k),
 \qquad
 \Delta F(k)=\tanh(\lambda/2)|\nabla F(k)|
 \le\frac{C}{a}|\nabla F(k)|.
\]
Moreover, $F(k)\ge g(b)\ge d/16$, so
$1/a\le4\sqrt{F(k)/(a^2d)}$. This also proves
\eqref{eq:second-difference-large-2} in this range.

At $k=\Theta-1$, truncation only decreases the forward increment
$u:=F(k+1)-F(k)$, leaving $v:=F(k)-F(k-1)>0$ unchanged.
Both $\nabla F(k)=(u+v)/2$ and
\[
 \frac{(\Delta F(k))_+}{\nabla F(k)}
 =\frac{(u-v)_+}{u+v}
\]
are nondecreasing in $u$. Thus the bounds just proved for the {\color{black}hyperbolic-cosine} potential remain valid here. At $k=\Theta$,
\[
 \Delta F(k)=-\nabla F(k),\qquad
 0\le\nabla F(k)=\frac{Y-F(k-1)}2
 \le\frac{e^\lambda-1}{2}F(k-1)\le\frac{C}{a}F(k).
\]
All differences vanish for $k>\Theta$. This completes the proof of
\eqref{eq:first-difference-bound}--\eqref{eq:second-difference-large-2}.

Finally, for $k\ge1$, let
$u:=F(k+1)-F(k)\ge0$ and $v:=F(k)-F(k-1)\ge0$.
Then
\[
 |\Delta F(k)|=\frac{|u-v|}{2}
 \le\frac{u+v}{2}=|\nabla F(k)|.
\]
At $k=0$, the bound follows from \eqref{eq:second-difference-small}.
This proves \eqref{eq:second-by-first}.
\end{proof}

\subsection{Bounding the number of exceptional rows}\label{sec:exceptional-count}

The next estimate quantifies the decrease available from the sign choice,
even when the row increments have very different sizes. The key idea is to decompose the change in potential into two separate scales. A fresh column support
samples at most one of the largest $\lfloor N/d\rfloor$ row increments in absolute value on
average, yielding the linear term in the estimate. The remaining row increments contribute through the square-root term.

\begin{lemma}[Sampling estimate]\label{lem:linear}
Let $U$ be a set of size $N$, let $R$ be a uniformly random $d$--subset of
$U$, and assume $2\le d\le N/2$.  Put $p:=d/N$ and
$r:=\lfloor1/p\rfloor$.  For arbitrary real numbers $(c_i)_{i\in U}$, let
$c_1^*\ge\cdots\ge c_N^*$ be the decreasing rearrangement of $(|c_i|)$.  Then
\begin{equation}\label{eq:linear-estimate}
 \EE\left|\sum_{i\in R}c_i\right|
 \ge c\left(
   p\sum_{j\le r}c_j^*
   +\left(p\sum_{j>r}(c_j^*)^2\right)^{1/2}
 \right)
\end{equation}
for an absolute constant $c>0$.
\end{lemma}

\begin{proof}
{
We first compare the sampled sum with the sampled Euclidean norm:
\begin{equation}\label{eq:sample-l2}
 \EE\left|\sum_{i\in R}c_i\right|
 \ge\frac14\,G,
 \qquad
 G:=\EE\left(\sum_{i\in R}c_i^2\right)^{1/2}.
\end{equation}
Choose a uniform $2d$--subset $Q$ of $U$, pair its elements uniformly at
random, and independently choose one endpoint of each pair. The resulting
set $R$ is a uniform $d$--subset of $U$. Conditional on $Q$, put
\[
 A:=\sum_{i\in Q}c_i,\qquad V^2:=\sum_{i\in Q}c_i^2,
 \qquad Z:=\sum_{\{u,v\}\text{ a pair}}(c_u-c_v)^2.
\]
For fixed pairs, the sampled sum has the form
$A/2+\frac12\sum_{\{u,v\}}\varepsilon_{uv}(c_u-c_v)$, with independent
uniform signs $\varepsilon_{uv}$. Jensen's inequality bounds its expected
absolute value below by $|A|/2$. Symmetry of the signed sum and
Khintchine's inequality give a second lower bound $\sqrt Z/(2\sqrt2)$.
Consequently,
\[
 \EE\left[\left|\sum_{i\in R}c_i\right|\mathrel{\Big|}Q\right]
 \ge\max\left\{\frac{|A|}{2},
                 \frac{\EE[\sqrt Z\mid Q]}{2\sqrt2}\right\}.
\]
Every pair of distinct elements of $Q$ is matched with probability
$1/(2d-1)$, so
\[
 Z\le2V^2,
 \qquad
 \EE[Z\mid Q]
 =\frac{1}{2d-1}\sum_{\{u,v\}\subset Q}(c_u-c_v)^2
 =\frac{2dV^2-A^2}{2d-1}.
\]
If $|A|\ge V$, the first lower bound is at least $V/2$. Otherwise,
$\EE[Z\mid Q]\ge V^2$ and
$\EE[\sqrt Z\mid Q]\ge\EE[Z\mid Q]/(\sqrt2 V)\ge V/\sqrt2$.
Thus the conditional expectation is always at least $V/4$.
Averaging and using $R\subset Q$ proves \eqref{eq:sample-l2}.

Relabel the coordinates so that $|c_j|=c_j^*$, and write
\[
 H:=p\sum_{j\le r}c_j^*,
 \qquad T^2:=p\sum_{j>r}(c_j^*)^2.
\]
For $J:=|R\cap[r]|$, Cauchy--Schwarz gives
\[
 \left(\sum_{i\in R}c_i^2\right)^{1/2}
 \ge\frac{\sum_{j\le r}c_j^*\ind_{\{j\in R\}}}{\sqrt J},
\]
where the right hand side is zero when $J=0$. For each $j\le r$,
\[
 \EE[J\mid j\in R]
 =1+\frac{(r-1)(d-1)}{N-1}\le2.
\]
Taking expectations and applying Jensen's inequality to $x^{-1/2}$
therefore gives $G\ge H/\sqrt2$.

It remains to bound $T$. We may assume $T>0$. If $c_{r+1}^*>T$, then
$H\ge prc_{r+1}^*\ge T/2$, since $pr\ge1/2$, and hence
$H+T\le3\sqrt2G$. Otherwise, put
$W:=\sum_{j>r}(c_j^*)^2\ind_{\{j\in R\}}$. The probability that two
distinct coordinates both belong to $R$ is at most $p^2$, whence
\[
 \EE W=T^2,
 \qquad
 \EE W^2\le T^4+p\sum_{j>r}(c_j^*)^4\le2T^4.
\]
H\"older's inequality now yields
\[
 G\ge\EE\sqrt W
 \ge\frac{(\EE W)^{3/2}}{(\EE W^2)^{1/2}}
 \ge\frac{T}{\sqrt2}.
\]
In this case $H+T\le2\sqrt2G$. Combining either bound with
\eqref{eq:sample-l2} proves \eqref{eq:linear-estimate}.
}
\end{proof}

Define
\begin{equation}\label{eq:Phi-def}
    \Phi_t:=Y|E_t|+\sum_{i\notin E_t}F(S_i(t)),
\end{equation}
so that
\begin{equation}\label{eq:E-by-Phi}
    |E_t|\le\Phi_t/Y.
\end{equation}
When a row first becomes exceptional, its one-row potential is exactly $Y$.
Consequently,
\begin{equation}\label{eq:Phi-increment}
 \Phi_t-\Phi_{t-1}
 =\sum_{i\notin E_{t-1}}
    \big(F(S_i(t))-F(S_i(t-1))\big).
\end{equation}
{
We will fix the absolute constant $K$ later. First consider columns
containing no exceptional row, for which the algorithm minimizes the potential.
}

\begin{lemma}[Columns containing no exceptional row]\label{lem:no-exceptional}
There are universal constants $C,c>0$ (independent of $a$ and $K$) with the following property.
If $a\ge C\sqrt K$ and $n$ is sufficiently large then, almost surely on the event
$\{\Phi_{t-1}\le Kn\}$,
\begin{equation}\label{eq:no-exceptional-drift}
 \EE[\Phi_t-\Phi_{t-1}\mid
      \calF_{t-1},R_t\cap E_{t-1}=\emptyset]
 \le\frac{C}{a^2}
 -c\,\EE\left[
   \left|\sum_{i\in R_t}\nabla F(S_i(t-1))\right|
   \mathrel{\Big|}\calF_{t-1},R_t\cap E_{t-1}=\emptyset
 \right].
\end{equation}
In particular, on this event the conditional drift is at most $C/a^2$.
\end{lemma}
\begin{proof}
All assertions below are made on the event $\{\Phi_{t-1}\le Kn\}$.
Condition on $\calF_{t-1}$, and let
\[
 E:=E_{t-1},\qquad U:=[n]\setminus E,\qquad N:=|U|.
\]
By \eqref{eq:E-by-Phi},
\[
 |E|\le\frac{\Phi_{t-1}}Y\le\frac{Kn}{Y},
\]
{
so $N\ge n/2$ and $d\le N/2$ for all sufficiently large $n$.
Conditional on $R_t\cap E=\emptyset$, which has positive probability,
$R_t$ is a uniform $d$--subset of $U$.
}

For $i\in U$, put
\[
    s_i:=S_i(t-1),\qquad v_i:=\nabla F(s_i),
    \qquad p:=d/N.
\]
By \eqref{eq:finite-difference-identity}, the minimizing rule
\eqref{eq:algorithm-potential-choice} gives the exact identity
\begin{equation}\label{eq:no-exceptional-exact}
    \Phi_t-\Phi_{t-1}
    =\sum_{i\in R_t}\Delta F(s_i)
      -\left|\sum_{i\in R_t}v_i\right|.
\end{equation}

Reorder the rows so that $|v_1|\ge\cdots\ge|v_N|$ and set
$r:=\lfloor1/p\rfloor$. By \eqref{eq:second-difference-small},
the rows with $|s_i|\le\rho/2$ contribute at most
$C/a^2$ to $p\sum_i(\Delta F(s_i))_+$.  For the remaining rows among the
first $r$, \eqref{eq:second-difference-large-1} gives a contribution at most
$(C/a)p\sum_{j\le r}|v_j|$.  For all remaining rows,
\eqref{eq:second-difference-large-2} and Cauchy's inequality give
\begin{align*}
 p\sum_{\substack{j>r\\ |s_j|>\rho/2}}
   (\Delta F(s_j))_+
 \quad\le C
  \left(p\sum_{j>r}v_j^2\right)^{1/2}
  \left(\frac{p}{a^2d}\sum_{i\in U}F(s_i)\right)^{1/2}
 \le\frac{C'\sqrt K}{a}
  \left(p\sum_{j>r}v_j^2\right)^{1/2}.
\end{align*}
Here we used $p=d/N$, $N\ge n/2$, and
$\sum_{i\in U}F(s_i)\le Kn$.  Since $K\ge1$, we have proved
\begin{equation}\label{eq:second-difference-bound}
 p\sum_{i\in U}(\Delta F(s_i))_+
 \le\frac{C}{a^2}+\frac{C\sqrt K}{a}
 \left(
   p\sum_{j\le r}|v_j|
   +\left(p\sum_{j>r}v_j^2\right)^{1/2}
 \right).
\end{equation}

The sampling estimate applies because, under the present conditioning,
$R_t$ is a uniform $d$--subset of $U$.  It gives
\begin{equation}\label{eq:linear-applied}
 \EE\left[\left|\sum_{i\in R_t}v_i\right|
   \mathrel{\Big|}\calF_{t-1},R_t\cap E=\emptyset\right]
 \ge c\left(
   p\sum_{j\le r}|v_j|
   +\left(p\sum_{j>r}v_j^2\right)^{1/2}
 \right).
\end{equation}
Taking expectations in \eqref{eq:no-exceptional-exact} and using
\eqref{eq:second-difference-bound} and \eqref{eq:linear-applied}, we obtain
the result under the assumption that
$a\ge C''\sqrt K$ for a sufficiently large absolute constant $C''>0$.
\end{proof}

{
When a column contains an exceptional row, the greedy rule overrides
potential minimization. The following estimate bounds the resulting cost.
}

\begin{lemma}[Columns containing an exceptional row]\label{lem:exceptional-column}
There is a universal constant $C>0$ with the following property.
Let $n$ be sufficiently large,
then almost surely on the event
\[
 \{\Phi_{t-1}\le Kn,\ E_{t-1}\ne\emptyset\},
\]
we have
\begin{equation}\label{eq:exceptional-column-cost}
 \EE[\Phi_t-\Phi_{t-1}\mid
      \calF_{t-1},R_t\cap E_{t-1}\ne\emptyset]
 \le \frac{CKd}{a}.
\end{equation}
\end{lemma}

\begin{proof}
{
Retain the notation $E,U,N$ from the preceding proof. Conditional on
$R_t\cap E\ne\emptyset$, the set $R_t\cap U$ is exchangeable over $U$
and has size less than $d$. Hence, for every $i\in U$,
\[
 \PP\{i\in R_t\mid\calF_{t-1},R_t\cap E\ne\emptyset\}
 \le\frac dN\le\frac{2d}{n}.
\]
By \eqref{eq:first-difference-bound} and \eqref{eq:second-by-first},
including the bound at $k=0$ and the identity $F(0)=1$,
\begin{equation}\label{eq:one-row-increment-bound}
 |\Delta F(k)|+|\nabla F(k)|\le\frac{C}{a}F(k)
 \qquad(k\in\ZZ).
\end{equation}
Consequently, \eqref{eq:Phi-increment} and
\eqref{eq:finite-difference-identity} give
\[
 \EE[\Phi_t-\Phi_{t-1}\mid\calF_{t-1},R_t\cap E\ne\emptyset]
 \le\frac{Cd}{an}\sum_{i\in U}F(S_i(t-1))
 \le\frac{CKd}{a}.
\]
}
\end{proof}

The next proposition bounds the fluctuations of the potential.  Together
with the two drift estimates, it will allow us to control its maximum by
Freedman's inequality.

\begin{proposition}[Potential increments]\label{prop:potential-increments}
For every fixed $K\ge1$, there is a constant $C_K>0$ such that, for all
sufficiently large $n$, almost surely on $\{\Phi_{t-1}\le Kn\}$,
\begin{equation}\label{eq:Phi-variance-increment}
 \EE[(\Phi_t-\Phi_{t-1})^2\mid\calF_{t-1}]\le C_KdY,
 \qquad
 |\Phi_t-\Phi_{t-1}|\le C_KdY.
\end{equation}
\end{proposition}

\begin{proof}
Condition on $\calF_{t-1}$ and assume $\Phi_{t-1}\le Kn$.  Set
\[
 g_i:=\big(|\Delta F(S_i(t-1))|+|\nabla F(S_i(t-1))|\big)
       \ind_{\{i\notin E_{t-1}\}}.
\]
By \eqref{eq:Phi-increment} and \eqref{eq:finite-difference-identity},
\[
 |\Phi_t-\Phi_{t-1}|\le\sum_{i\in R_t}g_i.
\]
{
By \eqref{eq:one-row-increment-bound},
$g_i\le CF(S_i(t-1))$ for $i\notin E_{t-1}$. Since $F\le Y$,
}
\[
 \sum_i g_i\le CK n,
 \qquad
 \sum_i g_i^2\le CY\sum_{i\notin E_{t-1}}F(S_i(t-1))
                  \le CK nY.
\]
For a uniform $d$--subset $R$ of $[n]$, two distinct coordinates are both
sampled with probability at most $(d/n)^2$.  Thus
\begin{equation}\label{eq:sampling-second-moment}
 \EE\left(\sum_{i\in R}g_i\right)^2
 \le\frac dn\sum_i g_i^2
      +\left(\frac dn\sum_i g_i\right)^2
 \le CKdY+C^2K^2d^2
 \le C_KdY,
\end{equation}
where the last inequality uses $Y\ge d$.  This proves the conditional
second-moment bound.  Finally, $g_i\le CY$ and $|R_t|=d$ give the
absolute increment bound.
\end{proof}

\begin{proposition}[Few rows reach $\Theta$]\label{prop:few-exceptional}
Uniformly for $2\le d\le {n^{1/5}}$,
\begin{equation}\label{eq:Phi-control}
    \PP\left\{\max_{t\le n}\Phi_t\le Kn\right\}=1-O(n^{-1}).
\end{equation}
Consequently, with probability $1-O(n^{-1})$,
\begin{equation}\label{eq:E-control}
    |E_t|\le\frac{Kn}{Y}
    \qquad\text{for every }t\le n.
\end{equation}
\end{proposition}

\begin{proof}
{
Choose absolute constants $C_0,C_1>0$ for the bounds in
\cref{lem:no-exceptional,lem:exceptional-column}. Fix $K\ge4$, and then
choose $a$ sufficiently large that \cref{lem:no-exceptional} applies and
$\beta:=2C_0/a^2\le1$. On $\{\Phi_{t-1}\le Kn\}$, a union bound gives
\begin{equation}\label{eq:bad-step-probability}
 \PP\{R_t\cap E_{t-1}\ne\emptyset\mid\calF_{t-1}\}
 \le\frac{d|E_{t-1}|}{n}\le\frac{Kd}{Y}.
\end{equation}
Combining the two drift estimates therefore yields
\[
 \EE[\Phi_t-\Phi_{t-1}\mid\calF_{t-1}]
 \le\frac{C_0}{a^2}+\frac{C_1K^2d^2}{aY}
 =\frac{C_0}{a^2}+\frac{C_1K^2}{ad(\log\log n)^6}
 \le\beta
\]
for all sufficiently large $n$, uniformly for $d\ge2$.

Stop at the first crossing of $Kn$, setting
\[
 \tau:=\min\big(\{t\in[n]:\Phi_t>Kn\}\cup\{n\}\big).
\]
Since $\{t\le\tau\}$ is $\calF_{t-1}$--measurable, the process
\begin{equation}\label{eq:Z-def}
 Z_t:=\Phi_{t\wedge\tau}-n-\beta(t\wedge\tau)
\end{equation}
is a supermartingale with $Z_0=0$. By
\cref{prop:potential-increments}, after enlarging $C_K$,
\begin{equation}\label{eq:Z-variance-increment}
 \EE[(Z_t-Z_{t-1})^2\mid\calF_{t-1}]\le C_KdY,
 \qquad
 |Z_t-Z_{t-1}|\le C_KdY.
\end{equation}
Freedman's inequality for supermartingales \cite{freedman} gives
\[
 \PP\{\max_{t\le n}Z_t\ge x\}
 \le\exp\left\{-\frac{x^2}{2(C_KndY+C_KdYx/3)}\right\}.
\]
Taking $x=Kn/2$, we obtain
\begin{equation}\label{eq:Freedman}
 \PP\left\{\max_{t\le n}Z_t\ge\frac{Kn}{2}\right\}
 \le\exp\left\{-c_K\frac{n}{dY}\right\}\le n^{-1},
\end{equation}
uniformly for $d\le n^{1/5}$ and all sufficiently large $n$, since
$n/(dY)\ge n^{1/5}/(\log\log n)^6$.

If $\max_{t\le n}\Phi_t>Kn$, then at the first crossing time
\[
 Z_\tau>Kn-n-\beta\tau\ge(K-2)n\ge Kn/2.
\]
Thus \eqref{eq:Freedman} proves \eqref{eq:Phi-control}, and
\eqref{eq:E-control} follows from \eqref{eq:E-by-Phi}.
}
\end{proof}

\subsection{Rows above the threshold}\label{sec:above-threshold}

For $k\ge0$ and $0\le t\le n$, define
\begin{equation}\label{eq:M-def}
    M_k(t):=\left\{i\in[n]:
       \max_{s\le t}|S_i(s)|\ge\Theta+3k\right\}.
\end{equation}
Thus $M_0(t)=E_t$.  The gap of three between successive levels is chosen 
for convenience.

\begin{lemma}[Three shared columns are necessary]\label{lem:witness}
Fix $k\ge1$ and $i\in M_k(n)$.  There are distinct times
$j_1<j_2<j_3$ such that, for each $u\in\{1,2,3\}$,
\begin{align*}
    i&\in R_{j_u},\\
    |S_i(j_u-1)|&\ge\Theta+3(k-1),\\
    |S_i(j_u)|&>|S_i(j_u-1)|,
\end{align*}
and
\begin{equation}\label{eq:witness-shared-column}
    \big(M_{k-1}(j_u-1)\setminus\{i\}\big)\cap R_{j_u}\ne\emptyset.
\end{equation}
\end{lemma}

\begin{proof}
Before $|S_i|$ first reaches $\Theta+3k$, it makes three upward moves from
states of height at least $\Theta+3(k-1)$.  At each such move, row $i$ is
already exceptional.  If the algorithm had selected $i$, its sign would have
decreased $|S_i|$.  It therefore selected another exceptional row in the
same column, and that row had absolute discrepancy at least
$|S_i(j_u-1)|$.  This is \eqref{eq:witness-shared-column}.
\end{proof}

\begin{lemma}[Recursion for successive levels]\label{lem:level-recursion}
For every $k\ge1$ and $\mu\in[0,1]$,
\begin{equation}\label{eq:level-expectation}
 \EE\left[
  |M_k(n)|\ind_{\{|M_{k-1}(n)|\le n\mu\}}
 \right]
 \le n(Cd^2\mu)^3.
\end{equation}
Consequently, for every $\mu'>0$,
\begin{equation}\label{eq:level-Markov}
 \PP\{|M_k(n)|>n\mu',\ |M_{k-1}(n)|\le n\mu\}
 \le\frac{(Cd^2\mu)^3}{\mu'}.
\end{equation}
\end{lemma}

\begin{proof}
Fix a row $i$ and three prescribed times $j_1<j_2<j_3$.  Conditional on
$\calF_{t-1}$, the set $M_{k-1}(t-1)$ is fixed and $R_t$ is fresh.  Whenever
$|M_{k-1}(t-1)|\le n\mu$, a union bound over the possible second row gives
\begin{align}\label{eq:one-shared-column}
 &\PP\left\{i\in R_t,\
   (R_t\setminus\{i\})\cap M_{k-1}(t-1)\ne\emptyset
   \mathrel{\Big|}\calF_{t-1}\right\}\nonumber\\
 &\qquad\le |M_{k-1}(t-1)|\frac{d(d-1)}{n(n-1)}
 \le\frac{Cd^2\mu}{n}.
\end{align}
To make the sequential conditioning explicit, intersect the witness event at
each time $j_u$ with the condition
$|M_{k-1}(j_u-1)|\le n\mu$.  This condition is known just before column
$j_u$ arrives, so iterating \eqref{eq:one-shared-column} bounds any prescribed
triple of such events by $(Cd^2\mu/n)^3$.  On the event
$|M_{k-1}(n)|\le n\mu$, nestedness makes all three added conditions automatic.
Hence Lemma~\ref{lem:witness} and a union bound over the fewer than $n^3$ choices
of witness times give
\begin{equation}\label{eq:single-row-level}
 \PP\{i\in M_k(n),\ |M_{k-1}(n)|\le n\mu\}
 \le C(d^2\mu)^3.
\end{equation}
Summing this estimate over $i\in[n]$ proves
\eqref{eq:level-expectation}, and Markov's inequality proves
\eqref{eq:level-Markov}.
\end{proof}

Fix a sufficiently large absolute constant $D$ and define {\color{black}the} sequence
\begin{equation}\label{eq:x-recursion}
 x_0:=\frac{DK}{(\log\log n)^6},
 \qquad
 x_k:=x_{k-1}^2\quad(k\ge1).
\end{equation}
We will prove that typically $|M_k(n)|\le n x_k/(Dd^3)$.  Notice that the assertion at
$k=0$ is exactly the bound $|E_n|\le Kn/Y$ from
Proposition~\ref{prop:few-exceptional}.

\begin{lemma}[Iteration of the recursion]\label{lem:iteration}
For every $r\ge0$,
\begin{equation}\label{eq:iteration-failure}
 \PP\left\{\max_{t\le n}\Phi_t\le Kn\ \text{ and }\
   |M_k(n)|>\frac{nx_k}{Dd^3}
   \text{ for some }k\le r\right\}
 \le\frac{C}{D^2}\sum_{k=0}^{r-1}x_k=O(x_0).
\end{equation}
\end{lemma}

\begin{proof}
On $\{\max_{t\le n}\Phi_t\le Kn\}$, the assertion at level zero follows
from \eqref{eq:E-by-Phi}.  If the estimates first fail at a level $k\ge1$,
apply \eqref{eq:level-Markov} with
\[
    \mu=\frac{x_{k-1}}{Dd^3},
    \qquad
    \mu'=\frac{x_k}{Dd^3}.
\]
Since $x_k=x_{k-1}^2$, this gives
\begin{align*}
\PP\left\{|M_k(n)|>\frac{nx_k}{Dd^3},\
              |M_{k-1}(n)|\le\frac{nx_{k-1}}{Dd^3}\right\}
\le
   \frac{\left(Cd^2x_{k-1}/(Dd^3)\right)^3}
        {x_k/(Dd^3)}
   =\frac{C}{D^2}x_{k-1}.
\end{align*}
Sum over the possible first failed levels.  Since
$x_k=x_0^{2^k}$ and $x_0\le1/2$ for large $n$, one has
$\sum_{k\ge0}x_k\le2x_0$.
\end{proof}

\begin{proof}[Proof of Proposition~\ref{prop:sparse-upper}]
Put $L:=\log\log n$.  By Proposition~\ref{prop:few-exceptional},
\[
 \PP\left\{\max_{t\le n}\Phi_t>Kn\right\}
 =O(n^{-1})
\]
uniformly for $d\le {n^{1/5}}$.  Also $x_0=DK/L^6=O(L^{-6})$.  Let
\[
    r_*:=\min\left\{k\ge0:\frac{nx_k}{Dd^3}<1\right\}.
\]
Since $x_k=x_0^{2^k}$, one has $r_*=O(\log\log n)$. By
Lemma~\ref{lem:iteration}, outside an event of probability $O(L^{-6})$,
$|M_k(n)|\le nx_k/(Dd^3)$ for every $k\le r_*$.  Hence
$M_{r_*}(n)=\emptyset$.
It follows that
\[
    \max_{t\le n}\max_{i\le n}|S_i(t)|<\Theta+3r_*.
\]
The second assertion in \eqref{eq:lambda-theta-bounds} completes the proof.
\end{proof}

\section{Proof of the main theorem}

For the remaining sparsities use the following

\begin{theorem}[Online Beck--Fiala \cite{altschuler-tikhomirov-bf}]
\label{thm:online-bf}
For every fixed $\eta>0$, there are constants $c_\eta,C_\eta>0$ and an
efficient randomized online algorithm such that, for every fixed sequence of
$T$ vectors $a_1,\ldots,a_T\in[-1,1]^n$ having at most $d$ nonzero coordinates
each,
\[
 \PP\left\{
   \max_{t\le T}\left\|\sum_{s\le t}\sigma_sa_s\right\|_\infty
   >C_\eta\sqrt d
 \right\}
 \le C_\eta T\exp\left\{-\frac{c_\eta d}{\log^{2+\eta}(ed)}\right\}.
\]
\end{theorem}

\begin{proof}[Proof of \cref{thm:main}]
{
The lower bound, including its extension to every prescribed $T\ge n$,
is \cref{prop:lower-horizons}.
}

For the upper bound, use Algorithm~\ref{alg:main} when $d\le {n^{1/5}}$.
Proposition~\ref{prop:sparse-upper} gives the desired estimate uniformly in this
range.  When $d>{n^{1/5}}$, apply \cref{thm:online-bf}.
\end{proof}


\begin{thebibliography}{99}
\bibitem{ALS2}
E.~Abbe, S.~Li, and A.~Sly.
\newblock Binary perceptron: efficient algorithms can find solutions in a rare
well-connected cluster.
\newblock In \emph{Proceedings of the 54th Annual ACM Symposium on Theory of
Computing}, pages 860--873, 2022.

\bibitem{ajtai-carpool}
M.~Ajtai, J.~Aspnes, M.~Naor, Y.~Rabani, L.~J.~Schulman, and O.~Waarts.
\newblock Fairness in scheduling.
\newblock \emph{Journal of Algorithms}, 29(2):306--357, 1998.


\bibitem{AT-old}
D.~J. Altschuler and K.~Tikhomirov.
\newblock A threshold for online balancing of sparse iid vectors.
\newblock \emph{arXiv preprint arXiv:2509.02432}, 2025.


\bibitem{altschuler-tikhomirov-bf}
D.~J. Altschuler and K.~Tikhomirov.
\newblock Online Beck--Fiala down to logarithmic sparsity.
\newblock \emph{arXiv preprint arXiv:2607.14238}, 2026.

\bibitem{self-balancing}
R.~Alweiss, Y.~P. Liu, and M.~Sawhney.
\newblock Discrepancy minimization via a self-balancing walk.
\newblock In \emph{Proceedings of the 53rd Annual ACM Symposium on Theory of
Computing}, pages 14--20, 2021.

\bibitem{bansal-online}
N.~Bansal, H.~Jiang, S.~Singla, and M.~Sinha.
\newblock Online vector balancing and geometric discrepancy.
\newblock In \emph{Proceedings of the 52nd Annual ACM Symposium on Theory of
Computing}, pages 1139--1152, 2020.

\bibitem{bansal-meka}
N.~Bansal and R.~Meka.
\newblock On the discrepancy of random low degree set systems.
\newblock \emph{Random Structures \& Algorithms}, 57(3):695--705, 2020.

\bibitem{bansal-spencer}
N.~Bansal and J.~H. Spencer.
\newblock On-line balancing of random inputs.
\newblock \emph{Random Structures \& Algorithms}, 57(4):879--891, 2020.

\bibitem{bansal-survey}
N.~Bansal.
\newblock Discrepancy theory and related algorithms.
\newblock In \emph{International Congress of Mathematicians, Vol.~7},
pages 5178--5210. EMS Press, 2023.

\bibitem{beck-fiala}
J.~Beck and T.~Fiala.
\newblock ``Integer-making'' theorems.
\newblock \emph{Discrete Applied Mathematics}, 3(1):1--8, 1981.

\bibitem{freedman}
D.~A. Freedman.
\newblock On tail probabilities for martingales.
\newblock \emph{Annals of Probability}, 3(1):100--118, 1975.

\bibitem{gupta-carpool}
A.~Gupta, R.~Krishnaswamy, A.~Kumar, and S.~Singla.
\newblock Online carpooling using expander decompositions.
\newblock In \emph{40th IARCS Annual Conference on Foundations of Software
Technology and Theoretical Computer Science}, 2020.

\bibitem{experimental-design}
C.~Harshaw, F.~S\"avje, D.~A. Spielman, and P.~Zhang.
\newblock Balancing covariates in randomized experiments with the
Gram--Schmidt walk design.
\newblock \emph{Journal of the American Statistical Association},
119(548):2934--2946, 2024.

\bibitem{kim-roche}
J.~H. Kim and J.~R. Roche.
\newblock Covering cubes by random half cubes, with applications to binary
neural networks.
\newblock \emph{Journal of Computer and System Sciences},
56(2):223--252, 1998.

\bibitem{optimal-online}
J.~Kulkarni, V.~Reis, and T.~Rothvoss.
\newblock Optimal online discrepancy minimization.
\newblock In \emph{Proceedings of the 56th Annual ACM Symposium on Theory of
Computing}, pages 1832--1840, 2024.

\bibitem{potukuchi-spectral}
A.~Potukuchi.
\newblock A spectral bound on hypergraph discrepancy.
\newblock In \emph{47th International Colloquium on Automata, Languages, and
Programming}, volume 168 of LIPIcs, Article 93, 2020.

\bibitem{random-hypergraph-discrepancy}
C.~MacRury, T.~Masa\v{r}\'ik, L.~Pai, and X.~P\'erez-Gim\'enez.
\newblock The phase transition of discrepancy in random hypergraphs.
\newblock \emph{SIAM Journal on Discrete Mathematics}, 37(3):1818--1841, 2023.

\bibitem{spencer-online}
J.~Spencer.
\newblock Balancing games.
\newblock \emph{Journal of Combinatorial Theory, Series B}, 23(1):68--74,
1977.

\bibitem{ALS1}
E.~Abbe, S.~Li, and A.~Sly.
\newblock Proof of the contiguity conjecture and lognormal limit for the symmetric perceptron.
\newblock In \emph{2021 IEEE 62nd Annual Symposium on Foundations of Computer Science}, pages 327--338, 2022.

\bibitem{aden-ali-online}
I.~Aden-Ali.
\newblock Optimal online discrepancy minimization in linear time.
\newblock \emph{arXiv preprint arXiv:2607.04388}, 2026.

\bibitem{dja-crit}
D.~J. Altschuler.
\newblock Critical window of the symmetric perceptron.
\newblock \emph{Electronic Journal of Probability}, 28, Paper No.~123, 2023.

\bibitem{dja-jnw}
D.~J. Altschuler and J.~Niles-Weed.
\newblock The discrepancy of random rectangular matrices.
\newblock \emph{arXiv preprint arXiv:2101.04036}, 2021.

\bibitem{APZ}
B.~Aubin, W.~Perkins, and L.~Zdeborov\'a.
\newblock Storage capacity in symmetric binary perceptrons.
\newblock \emph{Journal of Physics A: Mathematical and Theoretical}, 52(29):294003, 2019.

\bibitem{bansal-jiang2}
N.~Bansal and H.~Jiang.
\newblock Decoupling via affine spectral-independence: Beck--Fiala and Koml\'os bounds beyond Banaszczyk.
\newblock \emph{arXiv preprint arXiv:2508.03961}, 2025.

\bibitem{bansal-jiang1}
N.~Bansal and H.~Jiang.
\newblock An improved bound for the Beck--Fiala conjecture.
\newblock \emph{arXiv preprint arXiv:2508.01937}, 2025.

\bibitem{bansal-smooth1}
N.~Bansal, H.~Jiang, R.~Meka, S.~Singla, and M.~Sinha.
\newblock Prefix discrepancy, smoothed analysis, and combinatorial vector balancing.
\newblock In \emph{13th Innovations in Theoretical Computer Science Conference}, volume 215 of LIPIcs, Article 13, 2022.

\bibitem{bansal-smooth2}
N.~Bansal, H.~Jiang, R.~Meka, S.~Singla, and M.~Sinha.
\newblock Smoothed analysis of the Koml\'os conjecture.
\newblock In \emph{49th International Colloquium on Automata, Languages, and Programming}, volume 229 of LIPIcs, Article 14, 2022.

\bibitem{bin}
R.~Hoberg and T.~Rothvoss.
\newblock A logarithmic additive integrality gap for bin packing.
\newblock In \emph{Proceedings of the Twenty-Eighth Annual ACM-SIAM Symposium on Discrete Algorithms}, pages 2616--2625, 2017.

\bibitem{PX}
W.~Perkins and C.~Xu.
\newblock Frozen $1$-RSB structure of the symmetric Ising perceptron.
\newblock In \emph{Proceedings of the 53rd Annual ACM Symposium on Theory of Computing}, pages 1579--1588, 2021.

\bibitem{ss}
A.~Sah and M.~Sawhney.
\newblock Distribution of the threshold for the symmetric perceptron.
\newblock \emph{arXiv preprint arXiv:2301.10701}, 2023.

\bibitem{spencer1985six}
J.~Spencer.
\newblock Six standard deviations suffice.
\newblock \emph{Transactions of the American Mathematical Society}, 289(2):679--706, 1985.

\bibitem{matrixspencer-cosine}
A.~Zouzias.
\newblock A matrix hyperbolic cosine algorithm and applications.
\newblock In \emph{Automata, Languages, and Programming, Part~I},
volume 7391 of \emph{Lecture Notes in Computer Science},
pages 846--858. Springer, Heidelberg, 2012.

\end{thebibliography}
\end{document}